\documentclass[12pt,leqno]{amsart}
\usepackage{color}
\usepackage{euscript}
\usepackage{amssymb,amsfonts}
\usepackage[hypertex]{hyperref}

\swapnumbers
\newtheorem{theorem}{Theorem}[section]

\newtheorem{lemma}[theorem]{Lemma}
\newtheorem{proposition}[theorem]{Proposition}

\newtheorem*{theoremA}{Theorem~A}
\newtheorem*{theoremB}{Theorem~B}
\newtheorem*{theoremC}{Theorem~C}
\newtheorem*{theoremD}{Theorem~D}

\theoremstyle{definition}

\newtheorem{odstavec}[theorem]{}
\newtheorem{example}[theorem]{Example}
\newtheorem{remark}[theorem]{Remark}
\newtheorem{definition}[theorem]{Definition}

\newtheorem*{formality-theorem}{Formality}

\def\tE{{E}}

\DeclareMathOperator{\id}{id} \def\RRR{\mathbb R} \def\frakS{\mathfrak S}
 \def\deviace{{\mathfrak D}}
\def\Chi{{\EuScript X}}
\def\FF{{\EuScript F}}

 \def\bbbR{{\mathbb R}} \def\bbR{\bbbR}
\def\Tr{{\it Tr\/}}  \def\Rn{{\bbR^n}} \def\Con{{\it Con\/}}
\def\sqdots{. \hskip -.01em .\hskip -.01em .}
\def\deltah{\delta_{\rm h}}\def\Graphname{\Lambda}
\def\vforder{\hbox{${\it ord}_{\rm vf}$}}
\def\sgn{{\rm sgn}}
\def\deltav{\delta_{\rm v}}\def\Tot{{\rm Tot\/}} \def\Zh{Z_{\rm h}}
\def\cyclsum#1{{\raisebox{-.1em}{\Large $\circ$}\hskip-1.2em\sum_#1}}
\def\cyclsumm#1{{\raisebox{-.1em}{\Large $\circ$}\hskip-1.4em\sum_#1}}
\def\cyclsumvf#1{{\raisebox{-.1em}{\Large $\circ$}\hskip-1.5em\sum_#1}}
\def\textcyclsum{{\raisebox{-.05em}{\large $\circ$}\hskip-.85em\sum}}
\def\bbB{{\mathbb B}} \def\Ker{{\it Ker\/}}  
 \def\gO{{\EuScript O\/}} \def\vt{\vartheta}
 \def\Span{{\it Span\/}}
 
\def\Grd #1#2{\Gr^{#1}[#2](d)} \def\Kr{{\EuScript K}}
\def\Gr{{\EuScript {G}\rm r}} \def\Vert{{\it Vert\/}}
 \def\Krtor{{\EuScript K}}
\def\Cont{{\EuScript {C}\rm ont}}

\def\Rada#1#2#3{#1_{#2},\dots,#1_{#3}} \def\lot{{\it l.o.t.}}
 
\def\sbbox{{\raisebox {.1em}{\rule{.4em}{.4em}} \hskip .1em}} 
\def\pa{\partial}  
\def\Contor{{\it Con\/}}
\def\Nat{{\mathfrak {Nat}\/}} 
\def\rada#1#2{#1,\ldots,#2}
\def\prodrada#1#2{#1 \cdots #2}
\def\Rada#1#2#3{#1_{#2},\dots,#1_{#3}}

\def\iT{{\overline T}}

\def\iL{{\overline L}}

\def\cases#1#2#3#4{
                  \left\{
                         \begin{array}{ll}
                           #1,\ &\mbox{#2}
                           \\
                           #3,\ &\mbox{#4}
                          \end{array}
                   \right.
}

\def\anchor{\unitlength .25cm
\begin{picture}(1,1.4)(-1,-.7)
\put(-.45,.55){\makebox(0,0)[cc]{$\sbbox$}}
\put(-.5,-.8){\vector(0,1){1.2}}
\end{picture}}

\def\anteanchor#1{\unitlength .25cm
\begin{picture}(1.5,1)(-1,-.7)
\put(-.45,-.55){\makebox(0,0)[cc]{$\bullet$}}
\put(-.45,-.55){\makebox(0,0)[lt]{\hskip .2em \scriptsize $#1$}}
\put(-.5,-.2){\vector(0,1){1.2}}
\end{picture}}

\newcommand{\Gam}{\Gamma}
\newcommand{\lam}{\lambda}

\begin{document}

\title[Combinatorial differential geometry -- torsion case]{
COMBINATORIAL DIFFERENTIAL GEOMETRY AND IDEAL BIANCHI--RICCI
IDENTITIES II -- the torsion case
}

\begin{abstract}
This paper is a continuation of~\cite{janyska-markl}, dealing with a
general, not-necessarily torsion-free, connection.  It characterizes
all possible systems of generators for vector-field valued operators
that depend naturally on a set of vector fields and a linear
connection, describes the size of the space of such operators and
proves the existence of an `ideal' basis consisting of operators with
given leading terms which satisfy the (generalized) Bianchi--Ricci
identities without corrections.
\end{abstract}

\author[Jany\v ska, Markl]{J.~Jany\v ska, M.~Markl}

\thanks{The first author was supported
by the Ministry of Education of the Czech Republic under the Project
MSM0021622409 and by the grant GA \v CR 201/09/0981. The second author was supported 
   by the grant GA \v CR 201/08/0397 and by
   the Academy of Sciences of the Czech Republic,
   Institutional Research Plan No.~AV0Z10190503}

\address{
{\ }
\newline
Department of Mathematics and Statistics, Masaryk University
\newline
Kotl\'a\v{r}sk\'a 2, 611 37 Brno, The Czech Republic
\newline
E-mail: {\tt janyska@math.muni.cz}
\newline
{\ }
\newline
Mathematical Institute of the Academy 
\newline
{\v Z}itn{\'a} 25,
         115 67 Prague 1, The Czech Republic
\newline
E-mail: {\tt markl@math.cas.cz}
}



\maketitle

\bibliographystyle{plain}
\baselineskip18pt plus 1pt minus 1pt
\parskip3pt plus 1pt minus .5pt

\def\prodrada#1#2{#1 \cdots #2}
\def\pa{\partial} \def\Ker{{\it Ker\/}}
\def\Rada#1#2#3{#1_{#2},\dots,#1_{#3}}
\def\Kr{{\EuScript K}}
\def\rada#1#2{#1,\ldots,#2}
\def\lot{{\it l.o.t.}}
\def\iR{{\overline R}}
\def\iT{{\overline T}}
\def\iL{{\overline L}}
\def\sqdots{\ldots}
\def\cyclsumnic{{\raisebox{-.1em}{\Large $\circ$}\hskip-1.2em\sum}}
\def\cyclsumvf#1{{\raisebox{-.1em}{\Large
      $\circ$}\hskip-1.5em\sum_#1}}
\def\textcyclsum{{\raisebox{-.05em}{\large $\circ$}\hskip-.85em\sum}}

\def\today{January 19, 2011}
\catcode`\@=11
\def\@evenfoot{\rule{0pt}{20pt}[\today] \hfill{\tt \jobname.tex}}
\def\@oddfoot{\rule{0pt}{20pt}{\tt \jobname.tex}\hfill [\today]}

\noindent 
{\bf Methods of the paper} are based on the graph complex approach
developed in~\cite{markl:ig,markl:na}. Most of the proofs
in this paper are parallel to the proofs of the analogous statements for
the torsion-free case given
in~\cite{janyska-markl}.

\vskip .5em

\noindent 
{\bf Plan of the paper.}
In Section~\ref{s1} we recall the basis features of the torsion case
and quote the classical reduction theorem due to \L{}ubczonok~\cite{Lub72}. 
In Section~\ref{s2} we formulate the main results of the paper
(Theorems~A--D) and show some explicit
calculations. The difference from the torsion-free case is obvious
already in the formulation of Theorem~A. In contrast to the
corresponding~\cite[Theorem~A]{janyska-markl}, we allow the basis operators
to be indexed by a two-parameter set $S$ rather than just natural
numbers $n\geq 3$ as in the torsion-free case. We had to accept this
generality because the `classical' basis consist of two
families of operators -- the iterated covariant derivatives of the
curvature {\em and\/} the iterated covariant derivatives of the
torsion, see Subsection~\ref{a}.   

All proofs are contained in Section~\ref{proofs}. As they are parallel
to the proofs in the torsion-free case of~\cite{janyska-markl}, we had
two extremal choices -- either to give no proofs at all, saying that they
are `obvious' modifications of the proofs of~\cite{janyska-markl}, or
to modify the proofs of~\cite{janyska-markl} and include them in full
length.  We choose a compromise and included only proofs which are
`manifestly' different from the torsion-free case, namely those
dealing directly with the corresponding graph complex.

\vskip .5em
\noindent
{\bf Conventions:} At several places, the abbreviation \lot\ for
`lower order terms' is used. Its precise meaning will either be
explained or will be clear from the context.  We assume that this paper
is read in conjunction with~\cite{janyska-markl}, so we refer to that
article very often. We will however keep the formulation of the main
theorems self-consistent. 

\vskip .5em

\noindent
{\bf Notation:}
We will use notation parallel to that
of~\cite{janyska-markl}, the distinction against the torsion-free case
will be marked by the tilde~$\widetilde{(-)}$. For instance, while
$\Con$ denoted in~\cite{janyska-markl} the bundle of {\em
torsion-free\/} connections, here $\Con$ denotes the bundle of {\em
all\/} linear connections and $\widetilde\Contor$ the subbundle 
of torsion-free connections.

\section{Reduction theorems for non-symmetric connections}
\label{s1}

In this paper, $M$ will always denote a smooth manifold. The letters
\hbox{$X$, $Y$, $Z$, $U$, $V$,...,} with or without indexes, will denote
(smooth) vector fields on $M$. We also consider a linear (generally
non-symmetric) connection $\Gamma$ on $M$ with Christoffel symbols
$\Gamma^\lambda _{\mu \nu }$, $1 \leq \lambda ,\mu ,\nu \leq \dim(M)$,
see, for example,~\cite[Section III.7]{kobayashi-nomizu}. The symbol
$\nabla$ will denote the covariant derivative with respect to
$\Gamma$, and by $\nabla^{(r)}$ we will denote the sequence of
iterated covariant derivatives up to order $r$, i.e.  $\nabla^{(r)} =
(\id,\nabla,\dots,\nabla^r)$. The letter $R$ will denote the curvature
$(1,3)$-tensor field and the letter $T$ will denote the torsion
$(1,2)$-tensor field of $\Gamma $. In order to get formulas compatible
with the notation of our earlier paper~\cite{janyska-markl}, we assume
$R(X,Y)(Z)=\nabla_{[X,Y]}Z - [\nabla_X,\nabla_Y]Z$, i.e.~our curvature
tensor $R$ differs from the curvature tensor
of~\cite{kobayashi-nomizu} by the sign.

For non-symmetric connections we have (see, for 
example,~\cite[Section~III.5]{kobayashi-nomizu}) the first Bianchi identity 
\begin{equation}
\label{Eq: 1st B.i.}  
\cyclsumm {{X,Y,Z}} R(X,Y)(Z) = - \cyclsumm
{{X,Y,Z}} \big[(\nabla_{X} T)(Y,Z) + T(T(X,Y),Z)\big] \,,
\end{equation}
and the second Bianchi identity
\begin{equation}\label{Eq: 2nd B.i.}
\cyclsumm {{U,X,Y}} (\nabla_{U} R)(X,Y)(Z) = 
- \cyclsumm {{U,X,Y}}  R(T(U,X),Y)(Z) 
\,,
\end{equation}
where $\textcyclsum$ is the cyclic sum over the indicated
vector fields.
Further, if $\Phi$ is a $(1,r)$-tensor field, then we have the Ricci
identity
\begin{align}
\label{Eq: R.i.}
(\nabla_{X}\nabla_{Y}\Phi 
& - \nabla_{Y}\nabla_{X}\Phi)
(Z_1,\dots,Z_r) =
- R(X,Y)(\Phi(Z_1,\dots,Z_r))
\\
& \nonumber
+ \sum_{j=1}^r\Phi(Z_1,\dots,R(X,Y)(Z_j),\dots, Z_r) 
- (\nabla_{T(X,Y)} \Phi)(Z_1,\dots,Z_r)\,.
\end{align}

It is well-known, see, for
example,~\cite[Section~III.7]{kobayashi-nomizu}, 
that $\Gamma$ induces a torsion-free
connection $\widetilde\Gamma$ whose Christoffel symbols are obtained
by symmetrization of the Christoffel symbols of $\Gamma$.  Then
$\Gamma = \widetilde\Gamma + \tfrac12 T$ and we get
\begin{align}\label{Eq: 1.4}
R(X,Y)(Z)
& = 
\widetilde R(X,Y)(Z) - \tfrac 12 (\widetilde\nabla_X T)(Y,Z) 
+ \tfrac 12 (\widetilde\nabla_Y T)(X,Z)
\\
& \quad\nonumber
- \tfrac 14 T(X,T(Y,Z)) 
+ \tfrac 14 T(Y,T(X,Z))
- \tfrac 12 T(T(X,Y),Z)\,,
\end{align}
where $\widetilde R$ is the curvature of $\widetilde\Gam$ and
$\widetilde\nabla$ is the covariant derivative with respect to
$\widetilde\Gam$.  Further, $ \nabla_XY = \widetilde\nabla _X Y +
\tfrac 12 T(X,Y) $ which implies, for any $(1,r)$-tensor field $\Phi$,
\begin{align}\label{Eq: 1.5}
(\nabla_X\Phi)( & Y_1,\dots,Y_r) = 
(\widetilde\nabla_X\Phi) (Y_1,\dots,Y_r) + 
\\
&\quad\nonumber
+ \tfrac12T(X,\Phi(Y_1,\dots,Y_r))
- \tfrac12\sum_{j=1}^r\Phi(Y_1,\dots,T(X,Y_j),\dots,Y_r)
\end{align}
If we apply covariant derivatives on the identity~\eqref{Eq: 1.5}, we get
\begin{equation}\label{Eq: 1.6}
\nabla^{r} \Phi = \widetilde\nabla^{r} \Phi + l.o.t.\,,
\end{equation}
where $l.o.t.$ is a polynomial constructed from 
$\nabla^{(r-1)}\Phi$ and $\nabla^{(r-1)}T$. 
Especially, for the torsion tensor,
\begin{equation}\label{Eq: 1.7}
\nabla^{r} T = \widetilde\nabla^{r} T + l.o.t.\,.
\end{equation}
Similarly, from~\eqref{Eq: 1.4},
\begin{equation}\label{Eq: 1.8}
\nabla^{r}  R = \widetilde\nabla^{r}  \widetilde R + o.t.\,,
\end{equation}
where $o.t.$ is a polynomial constructed from 
$\nabla^{(r+1)}T$ and $\nabla^{(r-1)} R$. 

It is well-known, \cite[p.~91]{Veb27} 
and~\cite[p.~162]{Sch54}, that differential concomitants
(natural polynomial tensor fields in terminology of 
natural bundles \cite{kolar-michor-slovak,KruJan90, Nij52, Nij72,
terng:AMJ78}) depending on tensor fields and a torsion-free connection
can be expressed through given tensor fields, the curvature tensor
of given connection and their covariant derivatives. This result
is known as the first (operators on connections only) and the
second reduction theorems.   

Using the above splitting of connections with torsions into
the symmetric connections and the torsions, we can prove the reduction
theorem for connections with torsions, see \L{}ubczonok \cite{Lub72}.
Let us quote \L{}ubczonok's formulation of the {\em 
reduction theorem\/} for connections with torsions.

\begin{theorem}\label{Th: RT1}
If $\Omega$ is a differential concomitant of order $r$ of 
$\{\Phi_k\}_{k=1,\dots,s}$ and of the linear connection 
$\Gamma^\lambda_{\mu\nu}$ with torsion, then $\Omega$ 
is an (ordinary) concomitant of the quantities:
\begin{gather*}
\{\widetilde\nabla_{\kappa_l,\dots,\kappa_1}\Phi_k\}\,, 
\quad l=0,1,\dots,r\,,\,\,\,k=1,\dots,s\,,
\\
\{\widetilde\nabla_{\kappa_l,\dots,\kappa_1}  T^\lambda_{\mu\nu}\}\,, 
\quad l=0,1,\dots,r\,,
\\
\{\widetilde\nabla_{\kappa_1,\dots,\kappa_l}\widetilde R_\rho{}^\lam{}_{\mu\nu}\}\,, 
\quad l=0,1,\dots,r-1\,,
\end{gather*}
where $\widetilde R_\rho{}^\lam{}_{\mu\nu}$, $\widetilde \nabla$ denote 
the curvature tensor and the covariant derivative with respect to 
$\widetilde\Gamma^\lambda_{\mu\nu}$.
\end{theorem}

Formally, we can write $\Omega(\partial^{(r)}\Phi_k,\partial^{(r)}\Gamma)
= \widetilde\Omega(\widetilde\nabla^{(r)}\Phi_k, \widetilde\nabla^{(r)} T,
\widetilde\nabla^{(r-1)}\widetilde R)$.

\begin{remark}
The original \L{}ubczonok's result quoted above assumes the same
maximal order $r$ of derivatives of $\Phi_k$ and $\Gam$. But Theorem
\ref{Th: RT1} holds if the order with respect to $\Gamma$ is $(r-1)$
only, i.e. $\Omega(\partial^{(r)}\Phi_k,\partial^{(r-1)}\Gamma) =
\widetilde\Omega(\widetilde\nabla^{(r)}\Phi_k,
\widetilde\nabla^{(r-1)} T, \widetilde\nabla^{(r-2)}\widetilde R)$.
Theorem \ref{Th: RT1} is in fact valid for any order $s\ge r-1$ with
respect to $\Gam$, see, for example,~\cite{Jan04}.
\end{remark}

Thanks to the above relations (\ref{Eq: 1.6})--(\ref{Eq: 1.8})
between covariant derivatives with respect to $\Gamma$ and
$\widetilde\Gamma$, we can reformulate the reduction Theorem~\ref{Th:
  RT1} directly for connections with torsions.

\begin{theorem}\label{Th: RT2}
If $\Omega$ is a differential concomitant of order $r$ of 
$\{\Phi_k\}_{k=1,\dots,s}$ and of the linear connection 
$\Gamma^\lambda_{\mu\nu}$ with torsion, then $\Omega$ 
is an ordinary concomitant of the quantities:
\begin{gather*}
\{\nabla_{\kappa_l,\dots,\kappa_1}\Phi_k\}\,, 
\quad l=0,1,\dots,r\,,\,\,\,k=1,\dots,s\,,
\\
\{\nabla_{\kappa_l,\dots,\kappa_1}  T^\lambda_{\mu\nu}\}\,, 
\quad l=0,1,\dots,r\,,
\\
\{\nabla_{\kappa_1,\dots,\kappa_l} R_\rho{}^\lam{}_{\mu\nu}\}\,, 
\quad l=0,1,\dots,r-1\,,
\end{gather*}
i.e.
$$
\Omega(\partial^{(r)}\Phi_k,\partial^{(r)}\Gamma)
= \overline\Omega(\nabla^{(r)}\Phi_k, \nabla^{(r)} T,
\nabla^{(r-1)} R).
$$
\end{theorem}

\begin{remark}{
We get, from Theorem~\ref{Th: RT1} and Theorem~\ref{Th: RT2}, that
$(\nabla^{(r)}\Phi_k,\nabla^{(r)} T, \nabla^{(r-1)} R)$ and
\hbox{$(\widetilde\nabla^{(r)}\Phi_k, \linebreak
  \widetilde\nabla^{(r)} T, \widetilde\nabla^{(r-1)}\widetilde R)$}
form two systems of generators of differential concomitants of order
$r$ of $\{\Phi_k\}_{k=1,\dots,s}$ and of the linear connection
$\Gamma^\lambda_{\mu\nu}$ with torsion (in order $r$).  These two
systems of generators satisfy different identities.  For the system
$(\nabla^{(r)}\Phi_k,\nabla^{(r-1)} T, \nabla^{(r-2)} R)$ we have the
Bianchi and the Ricci identities \eqref{Eq: 1st B.i.}, \eqref{Eq: 2nd
  B.i.} and \eqref{Eq: R.i.}  (and their covariant derivatives), while
for the system of generators
$(\widetilde\nabla^{(r)}\Phi_k,\widetilde\nabla^{(r-1)} T,
\widetilde\nabla^{(r-2)}\widetilde R)$ we have the Bianchi and the
Ricci identities (and their covariant derivatives) for torsion-free
connections recalled, for instance, in~\cite[Section~2]{janyska-markl}.

It follows from the Ricci identity that we can take also the {\em
  symmetrized\/} covariant derivatives
$(\overset{S}{\nabla}{}^{(r)}\Phi_k,\overset{S}{\nabla}{}^{(r)} T,
\overset{S}{\nabla}{}^{(r-1)} R)$ and
$(\overset{S}{\widetilde\nabla}{}^{(r)}\Phi_k,
\overset{S}{\widetilde\nabla}{}^{(r)} T,
\overset{S}{\widetilde\nabla}{}^{(r-1)}\widetilde R)$ as two different
bases of differential concomitants of order $r$.  The Bianchi-Ricci
identities for such symmetric bases are, however, quite involved.  We
will prove, in Theorem~C, that there are bases whose elements satisfy
the "ideal" Bianchi-Ricci identities (with vanishing right hand sides)
similar to the ideal Bianchi-Ricci identities for symmetric
connections, \cite{janyska-markl}.}
\end{remark}

\section{Main results}
\label{s2}

\begin{odstavec}
{\bf Operators we consider.}
\label{types}
Let $\Contor$ be the natural bundle functor of linear, not necessarily
torsion-free, connections~\cite[Section~17.7]{kolar-michor-slovak} and
$T$ the tangent bundle functor.  We will consider natural differential
operators $\gO : \Contor \times T^{\otimes d} \to T$ acting on a
linear connection and $d$ vector fields, $d \geq 0$, which are linear
in the vector field variables, and which have values in vector
fields. We will denote the space of natural operators of this type by
$\Nat(\Contor \times T^{\otimes d},T)$.

To make the formulation of the main results of this paper
self-consistent, we recall almost verbatim some definitions
of~\cite{janyska-markl}. 
Define the {\em vf-order\/} (vector-field order)\label{huhu}
resp.\ the {\em c-order\/} (connection order)\label{haha}
of a differential operator $\gO : \Contor
\times T^{\otimes d} \to T$ as the order of $\gO$ 
in the vector field variables, resp.\ the connection variable.
\end{odstavec}

\begin{odstavec}{\bf Traces.} 
\label{traces}
Let $\gO$ be an operator acting on vector fields $X_1,\ldots,X_d$ and
a connection $\Gamma$, with values in vector fields.  Suppose that
$\gO$ is a linear order $0$ differential operator in $X_i$ for some $1
\leq i \leq d$. This means that the local formula $O(\Gamma,\Rada
X1d)$ for $\gO$ is a linear function of the coordinates of $X_i$ and
does not contain derivatives of the coordinates of $X_i$. In this
situation we define $\Tr_i(\gO) \in \Nat(\Contor \times T^{\otimes (d-1)},R)$ 
as the operator with values in the bundle $R$ of
smooth functions given by the local formula
\begin{eqnarray*}
\lefteqn{
\Tr_i(O)(\Gamma,\Rada X1{i-1},\Rada X{i+1}d) :=} \hskip 1em
\\ 
&& \hskip 1em
\mbox{Trace}(O (\Gamma,\Rada X1{i-1},-,\Rada X{i+1}d) : \Rn \to \Rn).
\end{eqnarray*}
Whenever we
write $\Tr_i(\gO)$ we tacitly assume that the trace makes sense, 
i.e.~that $\gO$ is a linear order $0$ differential
operator in $X_i$. 
\end{odstavec}

\begin{odstavec}
{\bf Compositions.} 
\label{comp}
Let $\gO' : \Contor \times T^{\otimes d'} \to T$ and $\gO'' : \Contor \times
T^{\otimes d''} \to T$ be operators as in~\ref{types}. Assume that
$\gO'$ is a linear order $0$ differential operator in $X_i$ for some
$1 \leq i \leq d'$. In this situation we define the {\em composition\/} $\gO'
\circ_i \gO'' : \Contor \times T^{\otimes (d'+ d'' -1)} \to T$ as the
operator obtained by substituting the value of the operator $\gO''$
for the vector-field variable $X_i$ of $\gO'$. As in~\ref{traces}, by
writing $\gO' \circ_i \gO''$ we signal that $\gO'$ is of order $0$
in $X_i$.
\end{odstavec}

\begin{odstavec}
{\bf Iterations.}
\label{iter}
By an {\em iteration\/} of differential
operators we understand applying a finite number of the following
`elementary' operations:
\begin{itemize}
\item[(i)] permuting the vector-fields inputs of a differential operator $\gO$,

\item[(ii)] taking the pointwise linear
combination $k' \cdot \gO' + k'' \cdot \gO''$, $k',k''\in \mathbb{R}$,

\item[(iii)] performing the composition $\gO' \circ_i \gO''$, and

\item[(iv)] taking the pointwise product $\Tr_{i}(\gO') \cdot \gO''$.
\end{itemize}

There are `obvious' relations between the above operations. The
operations $\circ_i$ in (iii) satisfy the `operadic' associativity and
compatibility with permutations in~(i), see properties (1.9) and
(1.10) in~\cite[Definition~II.1.6]{markl-shnider-stasheff:book}. Other
`obvious' relations are the commutativity of the trace, $\Tr_j(\gO'
\circ_i \gO'') = \Tr_i(\gO''\circ_j \gO')$ and its `obvious'
compatibility with permutations of~(i).
\end{odstavec}

We denote, for each $n\geq 2$,
by $\tE^0(n)$ the induced representation
\[
\tE^0(n) :=
\mbox{\rm Ind}^{\Sigma_n}_{\Sigma_{n-2} \times \Sigma_2} 
({\mathbf 1}_{\Sigma_{n-2}}\otimes {\mathbb R}[\Sigma_{2}]),
\]
where ${\mathbb R}[\Sigma_{2}]$ is the regular representation of
$\Sigma_2$ and ${\mathbf 1}_{\Sigma_{n-2}}$ the trivial representations of the
symmetric group $\Sigma_{n-2}$.  The space $\tE^0(n)$ expresses the
symmetries of the derivative
\begin{equation}
\label{eq:1}
\frac{\pa^{n-2}\Gamma^\omega_{\rho_{n-1}\rho_n}}
{\prodrada{\pa x^{\rho_1}}{\pa x^{\rho_{n-2}}}}, \ n \geq 2,
\end{equation}
of the Christoffel symbol $\Gamma^\lambda_{\mu\nu}$, which is totally
symmetric in the first $(n-2)$ indexes but, unlike the torsion-free
case, {\em not\/} in the last two ones.  Elements of $\tE^0(n)$ are linear
combinations
\begin{equation}
\label{konec}
\sum_{\sigma \in \Sigma'_n} \alpha_\sigma 
\cdot (1_{n-2}\otimes \id_2) \sigma,
\end{equation}
where $1_{n-2}\otimes \id_2 \in {{\mathbf 1}}_{n-2}\otimes
\bbR[\Sigma_2]$ is the generator, $\alpha_\sigma \in {\mathbb R}$, and
$\sigma$ runs over the set $\Sigma'_n$ of all permutations $\sigma \in
\Sigma_n$ such that $\sigma(1) < \cdots <\sigma(n-2)$.  We also denote
$\tE^1(n)$ be the trivial $\Sigma_n$-module ${\mathbf 1}_n$ and by
\[
\vartheta_\tE : \tE^0(n) \to \tE^1(n)
\]
the equivariant map that sends the generator $1_{n-2} \otimes \id_2 \in {{\mathbf
1}}_{n-2}\otimes \bbR[\Sigma_2]$ to $-1_n \in {{\mathbf 1}}_n$.\label{a1}
Analogously to the torsion-free case discussed in~\cite{janyska-markl}, 
the leading terms of the basis tensors are
parametrized by a choice of generators for the kernel $\Kr(n) \subset
\tE^0(n)$ of the map $\vartheta_\tE$. 

The first main theorem of the paper reads:

\begin{theoremA}
\label{A}
Let $D^i_n(\Gamma,\Rada X1n)$, $(n,i) \in S := \{n \geq 2,\ 1 \leq i
\leq k_n\}$, be
differential operators in $\Nat(\Contor \times T^{\otimes n},T)$ whose
local expressions are
\begin{equation}
\label{oolong}
D_n^{i,\omega}\left(\Gamma^\lambda_{\mu\nu},
\rada{X^{\delta_1}_1}{X^{\delta_n}_n}\right) = \sum_{\sigma \in
  \Sigma'_n} 
\alpha^i_{n,\sigma} \cdot 
X^{\rho_1}_{\sigma(1)}\cdots X^{\rho_n}_{\sigma(n)} 
\frac{\pa^{n-2}\Gamma^\omega_{\rho_{n-1}\rho_n}}
{\prodrada{\pa x^{\rho_1}}{\pa x^{\rho_{n-2}}}} + \lot
\end{equation}
where $\lot$ is an expression of differential order $< n-2$, and
$\{\alpha^i_{n,\sigma}\}_{\sigma \in \Sigma'_n}$ are real constants such that the
elements 
\[
\sum_{\sigma \in \Sigma'_n} \alpha^i_{n,\sigma} 
\cdot (1_{n-2}\otimes \id_2) \sigma,\ 1 \leq i \leq k_n,
\]
generate the $\Sigma_n$-module $\Kr(n)$ for each $n \geq 2$.

Let moreover $V_n(\Gamma,\Rada X1n)$,
$n \geq 1$, be  differential operators in 
$\Nat(\Contor \times T^{\otimes n},T)$ of the form
\[
V_n^\omega\left(\Gamma^\lambda_{\mu\nu},
\rada{X^{\delta_1}_1}{X^{\delta_n}_n}\right) =
X^{\rho_1}_1\cdots X^{\rho_{n-1}}_{n-1}  
\frac{\pa^{n-1} X_n^{\omega_n}}
{\pa x^{\rho_1}\cdots{\pa x^{\rho_{n-1}}}
}
 + \lot,
\]
where $\lot$ is an expression of differential order 
$< n-1$.  

Suppose that the operators $D^i_n(\Gamma,\Rada X1n)$ are of
vf-order~$0$ and $V_n(\Gamma,\Rada X1n)$ of order $0$ in
$\Rada X1{n-1}$.  Then each differential operator $\gO : \Contor
\times T^{\otimes d} \to T$ is an iteration, in the sense
of~\ref{iter}, of some of the operators $\{D^i_n\}_{(n,i) \in S}$ and $\{V_n\}_{n \geq 1}$.
\end{theoremA}

On manifolds of dimension $\geq 3$, each sequence
of operators that generates all operators in $\Nat(\Contor \times
T^{\otimes n},T)$ is of the form required by Theorem~A. We
leave the precise formulation of this 
modification of~\cite[Theorem~B]{janyska-markl} to the reader.  Let us
spell out two preferred choices of the leading term of the operators
$D^i_n(\Gamma,\Rada X1n)$ in Theorem~A.

\vskip .5em

\begin{odstavec}
{\bf The classical choice.}\label{a} 
In this case $k_2: = 1$ and $k_n : = 2$ for $n \geq 3$. We put, for $n
\geq 3$,
\begin{equation}
\label{eq:2}
r_n^\omega\left(\Gamma^\lambda_{\mu\nu},
\rada{X^{\delta_1}_1}{X^{\delta_n}_n}\right) :=
X^{\rho_1}_1\cdots X^{\rho_n}_n \frac{\pa^{n-3}}
{\prodrada{\pa x^{\rho_1}}{\pa x^{\rho_{n-3}}}}
\left(
\frac{\pa \Gamma^\omega_{\rho_{n-2}\rho_n}}{\pa x^{\rho_{n-1}}}
- \frac{\pa  \Gamma^\omega_{\rho_{n-1}\rho_n}}{\pa x^{\rho_{n-2}}}
\right)
\end{equation}
and, for $n \geq 2$, 
\begin{equation}
t^\omega_n\left(\Gamma^\lambda_{\mu\nu},
\rada{X^{\delta_1}_1}{X^{\delta_n}_n}\right) := 
X^{\rho_1}_1\cdots X^{\rho_n}_n \frac{\pa^{n-2}}
{\prodrada{\pa x^{\rho_1}}{\pa x^{\rho_{n-2}}}}
\left(
\Gamma^\omega_{\rho_{n-1}\rho_n}
- 
\Gamma^\omega_{\rho_{n}\rho_{n-1}}
\right).
\end{equation}
Then $t_2$ (resp.~$r_n$ and $t_n$ if $n\geq 3$) generate, in the sense
required by Theorem~A, the kernel $\Krtor(2)$
(resp.~$\Krtor(n)$). So any system of operators $D^1_n$ with the leading
term $t_n$, $n \geq 2$, and operators $D^2_n$ with the leading term
$r_n$, $n \geq 3$, satisfy the requirements of Theorem~A.

The reader certainly noticed that $r_n$'s (resp.~$t_n$'s) are the leading
terms of the iterated covariant derivatives of the curvature (resp.~the
torsion), see also Example~\ref{example1}. This explains why we called this
choice {\em classical\/}. The term $r_n$ has the following
symmetries:
\begin{itemize}
\item[(s1)] 
antisymmetry in $X_{n-2}$ and $X_{n-1}$,
\item[(s3)]
for $n \geq 4$, cyclic symmetry in $X_{n-3}$, $X_{n-2}$,
$X_{n-1}$, and
\item[(s4)]
for $n \geq 4$, total symmetry in $X_1,\ldots,X_{n-3}$,
\end{itemize}
so there is no symmetry (s2) of~\cite{janyska-markl} typical for the
torsion-free case. The term $t_n$~is
\begin{itemize}
\item[(t1)] 
antisymmetric in $X_{n-1}$ and $X_{n}$, and
\item[(t2)]
for $n \geq 3$, totally symmetric in $X_1,\ldots,X_{n-2}$.
\end{itemize}
The terms $r_n$ and $t_n$ are not independent but tied,
for $n \geq 3$, by the vanishing of the sum
\begin{equation}
\label{eq:3}
\cyclsumnic \left(\rule{0pt}{1em} r_n(\Gamma,\Rada
X1{n-3},X_a,X_b,X_c) + t_n(\Gamma,\Rada X1{n-3},X_a,X_b,X_c)\right) = 0,
\end{equation}
running over all cyclic permutations $\{a,b,c\}$ of the set
$\{{n-2},{n-1},{n}\}$. 
\end{odstavec}

\begin{odstavec}
{\bf The canonical choice.} Now $k_n := 1$ for all $n \geq 2$.  Let
$l^\omega_2(\Gamma) := \Gamma^\omega_{\rho_1\rho_2} -
\Gamma^\omega_{\rho_2\rho_1}$ and $l_n$ be, for $n
\geq 3$, given by the local formula 
\[
l_n^\omega\left(\Gamma^\lambda_{\mu\nu},
\rada{X^{\delta_1}_1}{X^{\delta_n}_n}\right) := 
X^{\rho_1}_1\cdots X^{\rho_n}_n \frac{\pa^{n-3}}
{\prodrada{\pa x^{\rho_1}}{\pa x^{\rho_{n-3}}}}
\left(6
\frac{\pa \Gamma^\omega_{\rho_{n-1}\rho_n}}{\pa x^{\rho_{n-2}}}
- \sum_{a,b,c}
\frac{\pa  \Gamma^\omega_{\rho_{a}\rho_b}}{\pa x^{\rho_{c}}}
\right)
\]
where $\{a,b,c\}$ runs over all permutations of
$\{\rho_{n-2},\rho_{n-1},\rho_n\}$.  We call the choice {\em
canonical\/} because it is given by the canonical
$\Sigma_n$-equivariant projection of $\tE^0(n) = \Krtor(n) \oplus {\mathbf
1}_n$ onto $\Krtor(n)$.  The system $\{l_n\}_{n \geq 2}$ enjoys the
following symmetries:
\begin{itemize}
\item[(l1)] 
$l_2(\Gamma,X_1,X_2)$ is antisymmetric in $X_1,X_2$ and, for $n \geq 3$, 
\[
\sum_\omega l_n(\Gamma,\Rada X1{n-3},X_{\omega(n-2)},X_{\omega(n-1)},X_{\omega(n)}) = 0, 
\] 
with the sum over all permutations $\omega$ of 
$\{{n-2},{n-1},n\}$,
\item[(l2)] 
for $n \geq 3$, total symmetry in $X_1,\ldots,X_{n-3}$,
\item[(l3)]
for $n \geq 4$, 
\[
\sum_\omega (-1)^{\sgn(\omega)} \cdot l_n(\Gamma,
\Rada X1{n-4},X_{\omega(n-3)},X_{\omega(n-2)},X_{\omega(n-1)},X_n) = 0,
\]
where $\omega$ runs over all permutations of $\{n-3,n-2,n-1\}$, and
\item[(l4)]
for $n \geq 4$,
\[
\sum_{\tau,\lambda} (-1)^{\sgn(\tau) + \sgn(\lambda)}\cdot 
l_n(\Gamma,\Rada X1{n-4},X_{\tau(n-3)},X_{\tau(n-2)},X_{\lambda(n-1)},X_{\lambda(n)}) = 0, 
\] 
with the sum over all permutations $\tau$ (resp.~$\lambda$)
of $\{n-3,n-2\}$ (resp.~of $\{n-1,n\}$).
\end{itemize}
\label{b}
\end{odstavec}

The following theorem specifies more precisely which of the basis
operators may appear in the iterative representation of operators
$\Contor \times T^{\otimes d} \to T$.

\begin{theoremB}
Assume that $\dim(M) \geq 2d-1$ and that $\{D^i_n\}_{(n,i)\in S}$, 
$\{V_n\}_{n \geq 1}$ be as in Theorem~A.
Let $\gO : \Contor \times T^{\otimes d} \to T$ be a differential operator
of the vf-order  $a \geq 0$. Then it has an iterative representation
with the following property. Suppose that an 
additive factor of this iterative 
representation of $\gO$ via $\{D^i_n\}_{(n,i) \in S}$ and
$\{V_n\}_{n \geq 2}$ contains
$\rada {V_{q_1}}{V_{q_t}}$, for some $\Rada q1t \geq 2$,  
$t \geq 0$. Then
\[
q_1+\cdots +q_t \leq a +t.
\]
In particular, if $\gO$ is of vf-order~$0$, it has an 
iterative representation that uses only $\{D_n\}_{(n,i) \in S}$.
\end{theoremB} 
 
Theorem~B implies the following two `reduction' theorems.
The first one uses the `classical' choice of the generators of the
kernels $\Krtor(n)$, $n \geq 2$.

\begin{theorem}
\label{t:1}
Let $R_n$, $n \geq 3$, be operators of the form
\[
R_n^\omega\left(\Gamma^\lambda_{\mu\nu},
\rada{X^{\delta_1}_1}{X^{\delta_n}_n}\right)
=
X^{\rho_1}_1\cdots X^{\rho_n}_n \frac{\pa^{n-3}}
{\prodrada{\pa x^{\rho_1}}{\pa x^{\rho_{n-3}}}}
\left(
\frac{\pa \Gamma^\omega_{\rho_{n-2}\rho_n}}{\pa x^{\rho_{n-1}}}
- \frac{\pa  \Gamma^\omega_{\rho_{n-1}\rho_n}}{\pa x^{\rho_{n-2}}}
\right) + \lot
\]
and  $T_n$, $n \geq 2$, operators of the form
\[
T_n^\omega\left(\Gamma^\lambda_{\mu\nu},
\rada{X^{\delta_1}_1}{X^{\delta_n}_n}\right)
=
X^{\rho_1}_1\cdots X^{\rho_n}_n \frac{\pa^{n-2}}
{\prodrada{\pa x^{\rho_1}}{\pa x^{\rho_{n-2}}}}
\left(
\Gamma^\omega_{\rho_{n-1}\rho_n}
- 
\Gamma^\omega_{\rho_{n}\rho_{n-1}}
\right) + \lot
\]
If $\dim(M) \geq 2d-1$, the all differential concomitants $\gO :
\Contor \times T^{\otimes d} \to T$ of the connection
$\Gamma^\kappa_{\mu\nu}$ (i.e.~operators of the vf-order $0$) are
ordinary concomitants of $\{R_n\}_{n \geq 3}$ and
$\{T_n\}_{n \geq 2}$.
\end{theorem}

The `canonical' choice of the generators of the kernels $\Krtor(n)$ leads to

\begin{theorem}
\label{t:2}
Let $L_2(X_1,X_2) :=  T(X_1,X_2)$ be the torsion and
$L_n$, for $n \geq 3$, be operators of the form
\[
L_n^\omega\left(\Gamma^\lambda_{\mu\nu},
\rada{X^{\delta_1}_1}{X^{\delta_n}_n}\right)
=
X^{\rho_1}_1\cdots X^{\rho_n}_n \frac{\pa^{n-3}}
{\prodrada{\pa x^{\rho_1}}{\pa x^{\rho_{n-3}}}}
\left(6
\frac{\pa \Gamma^\omega_{\rho_{n-1}\rho_n}}{\pa x^{\rho_{n-2}}}
- \sum_{a,b,c}
\frac{\pa  \Gamma^\omega_{\rho_{a}\rho_b}}{\pa x^{\rho_{c}}}
\right) + \lot,
\]
where the sum runs over all permutations $\{a,b,c\}$ of
$\{{n-2},{n-1},{n}\}$. 
If $\dim(M) \geq 2d-1$, then all differential concomitants $\gO :
\Contor \times T^{\otimes d} \to T$  of the connection
$\Gamma^\kappa_{\mu\nu}$ are ordinary concomitants of the tensors
$\{L_n\}_{n \geq 2}$.
\end{theorem}

\begin{example}
\label{example1}
Tensors required by the above theorems (and therefore also by
Theorem~A) exist. One may, for instance, take
\begin{equation}
\label{j}
\def\arraystretch{1.4}
\begin{array}{rl}
R_n(\Gamma,\Rada X1n) & \hskip -.5em :=
(\nabla^{n-3} R)({X_1}, \cdots ,X_{n-3})(X_{n-2},X_{n-1})(X_n),\ n
\geq 3,\ \mbox { and}
\\
T_n(\Gamma,\Rada X1n) & \hskip -.5em  := 
(\nabla^{n-2} T)({X_1}, \cdots ,X_{n-2})(X_{n-1},X_n),\ n \geq 2,
\end{array}
\end{equation}
where $R$ and $T$ are the curvature and torsion tensors, respectively.
For the operators $L_n$, $n \geq 3$, in Theorem~\ref{t:2}, one can take
\begin{align}
\nonumber 
L_n(\Gamma,\Rada X1n):= &-3R_n(\Gamma,\Rada X1n) - 
R_n(\Gamma,\Rada X1{n-3},X_{n-1},X_n,X_{n-2})
\\
\label{JJ}
&+R_n(\Gamma,\Rada X1{n-3},X_{n},X_{n-2},X_{n-1}) + 2T_n(\Gamma,\Rada X1n)
\\
\nonumber 
& -2T_n(\Gamma,\Rada X1{n-3},X_{n-1},X_n,X_{n-2}).
\end{align}
where $T_n$ and $R_n$ are as in~(\ref{j}). 
\end{example}

Observe that, while the choice~(\ref{j}) in Theorem~\ref{t:1}
represents operators via the iterated covariant derivatives of both the
curvature {\em and\/} the torsion, the choice~(\ref{JJ}) in
Theorem~\ref{t:2} packs both series into one. Recall the following
important definition of~\cite{janyska-markl}.

\begin{definition}
\label{debra}
We say that $\frakS \in \RRR [\Sigma_n]$ is a {\em quasi-symmetry\/} of
an operator $D^i_n$ in~(\ref{oolong}) if
\[
(\sum_{\sigma \in \Sigma_n} \alpha^i_{n,\sigma} \sigma )\frakS = 0 
\]
in the group ring $ \RRR [\Sigma_n]$.
We say that  $\frakS$ is a {\em symmetry\/} of $D^i_n$ if $D^i_n \frakS = 0$.
\end{definition}

A  quasi-symmetry  $\frakS$ of $D^i_n$, by definition, annihilates its 
leading term, therefore $D^i_n \frakS$ is an operator of
c-order $\leq (n-3)$ that does not use the derivatives of the vector
field variables. We can express this fact by writing
\begin{equation}
\label{dev}
D^i_n \frakS(\Gamma,X_1,\ldots,X_n) = 
\deviace^{i,\frakS}_n(\Gamma,X_1,\ldots,X_n),
\end{equation}
where $\deviace^{i,\frakS}_n \in \Nat(\Con \times T^{\otimes n},T)$
($\deviace$ abbreviating ``deviation'') is a degree $\leq n-3$
operator which is, by Theorem~B, an iteration of the operators $D^i_u$
with $2 \leq u \leq n-1$ (no $V_n$'s). By definition, $\frakS$ is a
symmetry of $D^i_n$ if and only if $\deviace^{i,\frakS}_n =0$.  We
explained in~\cite{janyska-markl} that~(\ref{dev}) offers a conceptual
explanation of the Bianchi and Ricci identities. As in the
torsion-free case, one can prove that the iterative presentation of
Theorem~A is unique up to the quasi-symmetries and the `obvious'
relations, see~\cite[Theorem~D]{janyska-markl} for a precise
formulation. The following theorem guarantees the existence of
``ideal'' tensors.

\begin{theoremC}
For each choice of the leading terms
\begin{equation}
\label{psano_v_Soulu}
\sum_{\sigma \in \Sigma'_n} \alpha^i_{n,\sigma} \cdot
X^{\rho_1}_{\sigma(1)}\cdots X^{\rho_n}_{\sigma(n)}
\frac{\pa^{n-2}\Gamma^\omega_{\rho_{n-1}\rho_n}} {\prodrada{\pa
x^{\rho_1}}{\pa x^{\rho_{n-2}}}},\ (n,i) \in S,
\end{equation}
where $S$ is of the same form as in Theorem~A, such that
\begin{equation}
\label{kolo}
\sum_{\sigma \in \Sigma'_n} \alpha^i_{n,\sigma}  =0
\end{equation}
for each $(n,i) \in S$, there exist `ideal' operators
$\{J^i_n\}_{(n,i) \in S}$ as in~(\ref{oolong}), for which all the
``generalized'' Bianchi-Ricci identities~(\ref{dev}) are satisfied
without the right hand sides. In other words, all quasi-symmetries, in
the sense of Definition~\ref{debra}, are actual symmetries of the
operators $\{J^i_n\}_{(n,i) \in S}$.
\end{theoremC}

Observe that~(\ref{kolo}) means that $\sum_{\sigma \in \Sigma'_n}
\alpha^i_{n,\sigma} \cdot (1_{n-2}\otimes \id_2) \sigma$ belongs to
the kernel $\Krtor(n)$, but, in contrast to Theorem~A, we do not
assume that the elements corresponding to~(\ref{psano_v_Soulu})
generate the kernel.

\noindent 
{\bf Ideal tensors.}  Theorem~C implies the existence of streamlined
versions of the tensors $\{R_n\}_{n \geq 3}$, $\{T_n\}_{n \geq 2}$ and
$\{L_n\}_{n \geq 2}$ for which the quasi-symmetries induced by the
symmetries (s1), (s3), (s4), (t1), (t2), (l1), (l2), (l3), (l4) and
equation~(\ref{eq:3}) given on
pages~\pageref{a}-\pageref{b} are actual symmetries.
So one has tensors $\iR_n$, $n \geq 3$, $\iT_n$, $n
\geq 2$ and $\iL_n$, $n \geq 2$, such that
\begin{equation}
\label{b1}
\iR_n\left(\Gamma,X_1,\sqdots,X_{n-2},X_{n-1},X_n\right)+
\iR_n\left(\Gamma,X_1,\sqdots,X_{n-1},X_{n-2},X_n\right)
= 0,
\end{equation}
\begin{equation}
\label{b3}
\cyclsum {\sigma}
\iR_n\left(
\Gamma,X_1,\ldots,X_{n-4},X_{\sigma(n-3)},X_{\sigma(n-2)},X_{\sigma(n-1)},X_n
\right)=0,\ n \geq 4,
\end{equation}
where $\textcyclsum$ is the cyclic sum over the indicated
indexes, and
\begin{equation}
\label{b4}
\iR_n\left(
\Gamma,X_{\omega(1)},\ldots,X_{\omega(n-3)},X_{n-2},X_{n-1},X_n
\right)
=\iR_n\left(
\Gamma,X_1,\ldots,X_n
\right),
\end{equation}
for each $n \geq 4$ and a permutation $\omega \in \Sigma_{n-3}$. The
tensors $\iT_n$ satisfy
\begin{equation}
\label{t1}
\iT_n\left(\Gamma,X_1,\sqdots,X_{n-2},X_{n-1},X_n\right)+
\iT_n\left(\Gamma,X_1,\sqdots,X_{n-2},X_{n},X_{n-1}\right)
= 0,
\end{equation}
and, for $n\geq 3$, also
\begin{equation}
\label{t2}
\iT_n\left(
\Gamma,X_{\omega(1)},\ldots,X_{\omega(n-2)},X_{n-1},X_n
\right)
=\iT_n\left(
\Gamma,X_1,\ldots,X_n
\right),
\end{equation}
for each permutation $\omega \in \Sigma_{n-2}$.
Moreover,
\begin{eqnarray}
\label{aa}
\cyclsum {\sigma}
\lefteqn{
\iR_n\left(
\Gamma,X_1,\ldots,X_{n-3},X_{\sigma(n-2)},X_{\sigma(n-1)},X_{\sigma(n)}
\right)=\hskip 2em}
\\ \nonumber 
&&\hskip 2em= - \ \cyclsum {\sigma}
\iT_n\left(
\Gamma,X_1,\ldots,X_{n-3},X_{\sigma(n-2)},X_{\sigma(n-1)},X_{\sigma(n)}
\right),
\end{eqnarray}
with the sums running over cyclic permutations $\sigma$ of $\{n-2,n-1,n\}$.

The tensor $\iL_2$ is antisymmetric. The tensors $\iL_n$ satisfy, for
$n \geq 3$,
\begin{equation}
\label{l1}
\sum_{\omega}
\iL_n\left(\Gamma,X_1,\sqdots,X_{n-3},X_{\omega({n-2})},
X_{\omega({n-1})},X_{\omega(n)}\right)
= 0,
\end{equation}
where $\omega$ runs over all permutations of $\{n-2,n-1,n\}$. For $n
\geq 4$ they also satisfy
\begin{equation}
\label{l2}
\sum_\omega (-1)^{\sgn(\omega)} \cdot \iL_n(\Gamma,
\Rada X1{n-4},X_{\omega(n-3)},X_{\omega(n-2)},X_{\omega(n-1)},X_n) = 0,
\end{equation}
where $\omega$ runs over all permutations of $\{n-3,n-2,n-1\}$, 
\begin{equation}
\label{l3}
\iL_n\left(
\Gamma,X_{\omega(1)},\ldots,X_{\omega(n-3)},X_{n-2},X_{n-1},X_n
\right)
=\iL_n\left(
\Gamma,X_1,\ldots,X_n
\right),
\end{equation}
for each permutation $\omega \in \Sigma_{n-3}$, and
\begin{equation}
\label{l4}
\sum_{\tau,\lambda} (-1)^{\sgn(\tau) + \sgn(\lambda)}\cdot 
\iL_n(\Gamma,\Rada X1{n-4},
X_{\tau(n-3)},X_{\tau(n-2)},X_{\lambda(n-1)},X_{\lambda(n)}) = 0, 
\end{equation}
with the sum over all permutations $\tau$ (resp.~$\lambda$)
of $\{n-3,n-2\}$ (resp.~of $\{n-1,n\}$).

In Examples~\ref{Ex2}--\ref{Ex4} below we explicitly calculate the
ideal tensors $\iR_n$, $\iT_n$ and $\iL_n$ for $n \leq 4$. Our
calculation is facilitated by the following lemma whose
straightforward though technically involved proof we omit.

\begin{lemma}
\label{pisu_v_Ratajich}
Let $n \geq 3$ and $\Chi$, $V$ be vector spaces over a field of
characteristic zero. Denote by $\FF_L$ the space of all linear maps $L :
\Chi^{\otimes n} \to V$ with symmetry~(\ref{l1}) and, if $n \geq 4$,
also~(\ref{l2})--(\ref{l4}). Denote further 
by $\FF_{(R,T)}$ the space of all pairs
$(R,T)$ of linear maps $R,T : \Chi^{\otimes n} \to V$
satisfying~(\ref{b1}), (\ref{t1})--(\ref{aa}) and, if $n \geq 4$,
also~(\ref{b3}) and~(\ref{b4}).
Define finally the map $\Phi = (\Phi_R,\Phi_T): \FF_L \to \FF_{(R,T)}$ by 
\[
\begin{aligned}
\Phi_R(\Rada X1n) &:= \frac16\big[
L(\Rada X1{n-3},X_{n-1},X_{n-2},X_n) - L(\Rada X1n)\big], \mbox { and}
\\
\Phi_T(\Rada X1n) &:= \frac16\big[
L(\Rada X1n) - 
L(\Rada X1{n-2},X_n,X_{n-1}) \big],
\end{aligned}
\]
and the map
$\Psi : \FF_{(R,T)} \to \FF_L$ by
\[
\begin{aligned}
\Psi(\Rada X1n):= &-3R(\Rada X1n) - R(\Rada X1{n-3},X_{n-1},X_n,X_{n-2})
\\
&+R(\Rada X1{n-3},X_{n},X_{n-2},X_{n-1}) + 2T(\Rada X1n)
\\
& -2T(\Rada X1{n-3},X_{n-1},X_n,X_{n-2}).
\end{aligned}
\]
Then $\Phi$ and $\Psi$ are well-defined mutual inverses, $\Phi : \FF_L
\cong \FF_{(R,T)}: \Psi$.
\end{lemma}

The maps $\Phi$ and $\Psi$ of Lemma~\ref{pisu_v_Ratajich} produce from
ideal tensors $\iR_n, \iT_n$ the ideal tensor $\iL_n$ and vice
versa. Since the ideal tensors $\iR_n, \iT_n$ can be constructed as 
modification of the 
covariant derivatives of the classical curvature and torsion tensors,
we start in examples below with them and obtain $\iL_n$ as $\Psi(\iR_n,\iL_n)$.

\begin{example}
\label{Ex2}
If $n=2$, the tensor $\iT_2=T_2 = T$ satisfies the antisymmetry~(\ref{t1}),
so $\iL_2 = \iT_2 =  T$. There is, of course, no $\iR_2$.
\end{example}

To make formulas shorter, in the following two
examples we drop the implicit $\Gamma$ from the notation.

\begin{example}
\label{Ex3}
If $n=3$, then the tensor $R_3 = R$ satisfies~(\ref{b1}), the tensor 
$T_3 = \nabla T$
satisfies~(\ref{t1}) and, trivially, also~(\ref{t2}), but the couple
$(R_3,T_3) = (R,\nabla T)$ {\em does not\/} satisfy~(\ref{aa}).
If one takes, instead of $T_3 = \nabla T$, a streamlined version 
\[
\iT_3(X,Y,Z) := (\nabla_XT)(Y,Z) - T(X,T(Y,Z)), 
\]
then $\iT_3$ satisfies~(\ref{t1}),~(\ref{t2}), and the couple
$(\iR_3 = R_3,\iT_3)$ satisfies~(\ref{aa}) which is in this case 
precisely the first Bianchi 
identity \eqref{Eq: 1st B.i.}
for the curvature of a connection with nontrivial torsion.

It follows from Lemma~\ref{pisu_v_Ratajich} that 
the tensor $\iL_3$ defined by
\begin{equation}
\label{Eq: L3}
\begin{aligned}
\iL_3(X,Y,Z) 
 := & 
- 3\, \iR_3(X,Y,Z) - \iR_3(Y,Z,X) 
+ \iR_3(Z,X,Y)
\\
&+  2\, \iT_3(X,Y,Z) - 2 \, \iT(Y,Z,X)
\end{aligned}
\end{equation}
satisfies~(\ref{l1}), so it is the `ideal' $\iL_3$.
On the other hand, by the same lemma,
given $\iL_3$ satisfying~\eqref{l1}, we have
\begin{equation}\label{Eq: R3}
\iR_3(X,Y,Z) = - \frac 16 \big[\iL_3(X,Y,Z) 
- \iL_3(Y,X,Z)\big] 
\end{equation}
satisfying~\eqref{b1}. Further
\begin{align}\label{Eq: T3}
\iT_3(X,Y,Z) & = \frac 1{6} \big[\iL_3(X,Y,Z) - \iL_3(X,Z,Y)\big] 
\,
\end{align}
satisfies~\eqref{t1} and, trivially,~\eqref{t2}. Moreover, the pair
$(\iR_3,\iT_3)$ satisfies~\eqref{aa}.  If we put $\iR_3$ and $\iT_3$
calculated from \eqref{Eq: R3} and \eqref{Eq: T3} into \eqref{Eq: L3},
we recover $\iL_3$. Likewise, if we substitute $\iL_3$ calculated
from~\eqref{Eq: L3} into \eqref{Eq: R3} and \eqref{Eq: T3}, we 
get $\iR_3$ and $\iT_3$, because the transformations~(\ref{Eq: L3})
and~(\ref{Eq: R3})--(\ref{Eq: T3}) are, by
Lemma~\ref{pisu_v_Ratajich}, mutually inverse. 
\end{example}

\begin{example}
\label{Ex4}
If $n=4$, the tensor $R_4 = \nabla R$ satisfies ~(\ref{b1}) and,
trivially, also~(\ref{b4}) but {\it does not} satisfy ~(\ref{b3})
because of the non vanishing right hand side of the 2nd Bianchi
identity~\eqref{Eq: 2nd B.i.}.  We found the following explicit
formula for a streamlined couple $(\iR_4,\iT_4)$ in which $\iR_4$ is
given by
\begin{align}
\nonumber 
\iR_4&(X_1,\dots,X_4)
=
(\nabla_{X_1} R)(X_2,X_3)(X_4)
\\
& + \nonumber 
\frac12\,\big[
R(T(X_1,X_2),X_3)(X_4)
+ R(X_2,T(X_1,X_3))(X_4)\big]
\\
& - \nonumber
 \frac12\,\big[
T(R(X_2,X_3)(X_1),X_4) 
+ T((\nabla_{X_1} T)(X_2,X_3),X_4)
 +
T(T(T(X_2,X_3),X_1),X_4)
\big]
\\
& + \nonumber
 \frac14\,\big[
- 2\, (\nabla_{X_1} T)(T(X_2,X_3),X_4) 
- (\nabla_{X_2} T)(T(X_1,X_3),X_4)
+
(\nabla_{X_3} T)(T(X_1,X_2),X_4)  
\big]
\\
& + \nonumber
 \frac18\,\big[
T(T(X_3,X_4),T(X_1,X_2))
- T(T(X_2,X_4),T(X_1,X_3))
+2\, T(T(X_2,X_3),T(X_1,X_4))
\big].
\end{align}
It satisfies identities \eqref{b1},~\eqref{b3} and, trivially,
also~(\ref{b4}).
For $\iT_4$ we found
\begin{align*}
\iT_4   (X_1,X_{2},X_{3},X_{4})
&=
\frac 12\big[  
(\nabla_{X_1} \nabla_{X_{2}} T)(X_{3},X_{4}) 
+(\nabla_{X_{2}} \nabla_{X_1} T)(X_{3},X_{4})
\big]
\\
&
-\frac14 \big[
 R(X_1,X_{3})(T(X_{4},X_{2}))
+ R(X_2,X_{3})(T(X_{4},X_{1}))
\\
&\hskip 2cm
- R(X_1,X_{4})(T(X_{3},X_{2}))
- R(X_2,X_{4})(T(X_{3},X_{1}))
\big]
\\
&
+\frac34 \big[
 (\nabla_{X_1} T)(T(X_2,X_3),X_{4})
+ (\nabla_{X_2} T)(T(X_1,X_3),X_{4})
\\
&\hskip 2cm
- (\nabla_{X_1} T)(T(X_2,X_4),X_{3})
- (\nabla_{X_2} T)(T(X_1,X_4),X_{3})
\big]
\\
&
+\frac12 \big[
 T((\nabla_{X_1} T)(X_2,X_3),X_{4})
+ T((\nabla_{X_2} T)(X_1,X_3),X_{4})
\\
&\hskip 2cm -
T((\nabla_{X_1} T)(X_2,X_4),X_{3})
- T((\nabla_{X_2} T)(X_1,X_4),X_{3})
\big].
\end{align*}
It is easy to see that $\iT_4$ satisfies identities
\eqref{t1} and \eqref{t2}, and the pair $(\iR_4\,,\,\iT_4)$
satisfies~\eqref{aa}. 
By Lemma~\ref{pisu_v_Ratajich}, we may put
\begin{align*}
\iL_4& \left(  X_{{1}},X_{{2}},X_{{3}},X_{4}\right)
 =
- 3\,\iR_4( X_{{1}},X_{{2}}, X_{{3}},X_{4})
- \, \iR_4(X_{{1}},X_{{3}}, X_{{4}},X_{2})
\\
& \nonumber
\iT_4(X_{{1}},X_{{4}}, X_{{2}},X_{3})
+ 2 \, \iT_4(X_{{1}},X_{{2}}, X_{{3}},X_{4})
- 2 \, \iT_4(X_{{1}},X_{{3}}, X_{{4}},X_{2})\,.
\end{align*}

On the other hand, given an `ideal' $\iL_4$
satisfying~\eqref{l1}--\eqref{l4}, the equations
\begin{align*}
\iR_4(X_{{1}},X_{{2}},X_{{3}},X_{4})
& := 
\frac16\big[\iL_4(X_{{1}},X_{{3}},X_{{2}},X_{4}) 
-  \iL_4(X_{{1}},X_{{2}},X_{{3}},X_{4})\big]\ \mbox { and }
\\
\iT_4(X_{{1}},X_{{2}},X_{{3}},X_{4})
& := 
\frac16\big[\iL_4(X_{{1}},X_{{2}},X_{{3}},X_{4}) 
-  \iL_4(X_{{1}},X_{{2}},X_{{4}},X_{3})\big]
\end{align*}
determine `ideal' $\iR_4$ and $\iL_4$. 
\end{example}

We saw above that calculating the ideal tensors $\iR_n,\iT_n$ and
$\iL_n$ is difficult already for $n=4$. To find explicit
formulas for arbitrary $n\geq 3$ is, as in the torsion-free
case~\cite{janyska-markl}, a~challenging task.

Let $\Kr$ be the collection of the kernels~(\ref{kernel}) and
$\Gr[\Krtor](d)$ the space spanned by graphs with $d$ black
vertices~(\ref{black}), one vertex \anchor\ and a finite number of
vertices decorated by elements of~$\Kr$, see
pages~\pageref{proofs}--\pageref{k} of Section~\ref{proofs} for a precise
definition. 
The size of the space of natural operators 
$\Contor \times T^{\otimes d} \to T$ is described in:

\begin{theoremD}
On manifolds of dimension $\geq 2d-1$, the vector space $\Nat(\Contor
\times T^{\otimes d},T)$ is isomorphic to the graph space $\Gr[\Krtor](d)$.
\end{theoremD}

\begin{example}
As in the torsion-free case, the calculation of the dimension of
$\Gr[\Kr](d)$ is a purely combinatorial problem. For $d=1$ we get
$\dim(\Gr[\Kr](d)) = 1$, with the corresponding natural operator the
identity $X \mapsto X$.

One sees that, on manifolds of dimension $\geq 3$, 
$\dim(\Nat(\Con \times T^{\otimes 2},T)) = 7$. The
corresponding operators are 
\[
\nabla_XY,\ \nabla_YX,\ 
\Tr (\nabla_- Y) \cdot X\ \mbox {and}\  \Tr (\nabla_- X)\cdot Y
\]
as in the torsion-free case (see~\cite[Example~3.18]{janyska-markl}),
plus three operators
\[
T(X,Y),\ \Tr(T(-,Y)) \cdot X\
\mbox{and}\  \Tr(T(-,X)) \cdot Y
\]
involving the torsion.
\end{example}

\section{Proofs}
\label{proofs}
\setcounter{equation}{0}

As everywhere in this paper, we use the notation parallel to that
of~\cite{janyska-markl}, but the reader shall keep in mind that we
dropped the torsion-free assumption.  As expected, the proofs will be
based on a suitable graph complex describing operators of a given type
which was, in fact, already been described in~Section~4
of~\cite{janyska-markl}, see~4.7 of that section in particular. We
only briefly recall its definition, leaving the details and
motivations to~\cite{janyska-markl} and~\cite{markl:na}.

We consider the graded graph complex
$\Gr^*(d)$ whose
degree $m$ part $\Gr^m(d)$ is spanned by oriented graphs with precisely $d$
`black' vertices
\begin{equation}
\label{black}
\raisebox{-3em}{\rule{0pt}{0pt}}
b_u:= \hskip -6em 
\unitlength 4mm
\linethickness{0.4pt}
\begin{picture}(20,4)(10.5,18.8)
\put(20,20){\vector(0,1){2}}
\put(18,18){\vector(1,1){1.9}}
\put(22,18){\vector(-1,1){1.9}}
\put(19,18){\vector(1,2){.935}}
\put(20,20){\makebox(0,0)[cc]{\Large$\bullet$}}
\put(25,20){\makebox(0,0)[cc]{,\ $u \geq 0$,}}
\put(20,18){\makebox(0,0)[cc]{\scriptsize (\hskip 18mm)}}
\put(20.5,18){\makebox(0,0)[cc]{$\ldots$}}
\put(20,17){\makebox(0,0)[cc]{%
   $\underbrace{\rule{16mm}{0mm}}_{\mbox{\scriptsize $u$ inputs}}$}}
\end{picture}
\end{equation}
labelled $1,\ldots,d$,  some number of `$\nabla$-vertices' 
\begin{equation}
\label{nabla}
\unitlength 5mm
\begin{picture}(.7,2)(-.5,-.5)
\put(-.5,0){\makebox(0,0)[cc]{\Large$\nabla$}}
\put(2.5,0){\makebox(0,0)[cc]{, $u \geq 0$.}}
\put(-.5,.4){\vector(0,1){1}}
\put(.1,-1.25){\vector(-1,2){.5}}
\put(.59,-1.22){\vector(-1,1){.97}}
\unitlength 3mm
\put(-21,-20){
\put(19.4,18){\vector(1,4){0.475}}
\put(18.9,18){\vector(1,2){.95}}
\put(16,18){\vector(2,1){3.8}}
\put(17.75,18){\makebox(0,0)[cc]{$\ldots$}}
\put(17.7,16.6){\makebox(0,0)[cc]{%
   $\underbrace{\rule{10mm}{0mm}}_{\mbox{\scriptsize $u$ inputs}}$}}
\put(17.7,17.8){\makebox(0,0)[cc]{\scriptsize $( \hskip 12mm )$}}
}
\end{picture}
\raisebox{-2.5em}{\rule{0pt}{0pt}}
\end{equation}
precisely $m$ `white' vertices
\begin{equation}
\label{white}
\raisebox{-3em}{\rule{0pt}{0pt}}
\unitlength 4mm
\linethickness{0.4pt}
\begin{picture}(20,4)(10.5,18.8)
\put(20,20.3){\vector(0,1){1.7}}
\put(18,18){\vector(1,1){1.8}}
\put(22,18){\vector(-1,1){1.8}}
\put(19,18){\vector(1,2){.91}}
\put(20,20){\makebox(0,0)[cc]{\Large$\circ$}}
\put(25,20){\makebox(0,0)[cc]{,\ $u \geq 2$,}}
\put(20,18){\makebox(0,0)[cc]{\scriptsize (\hskip 18mm)}}
\put(20.5,18){\makebox(0,0)[cc]{$\ldots$}}
\put(20,17){\makebox(0,0)[cc]{%
   $\underbrace{\rule{16mm}{0mm}}_{\mbox{\scriptsize $u$ inputs}}$}}
\end{picture}
\end{equation}
and one vertex \anchor\ (the anchor). We will usually omit the
parentheses $( \hskip 1em )$ indicating that the inputs they encompass
are fully symmetric.  In contrast to the torsion-free case, the
$\nabla$-vertex~(\ref{nabla}) is \underline{not} symmetric in the
rightmost two inputs. The interpretation of the vertices is explained
in~\cite[Section~4]{janyska-markl}.  The differential is given by the
replacement rules that are `informally' the same as these
in~\cite[Section~4]{janyska-markl} (but formally not, since the
symmetries of the $\nabla$-vertex are different), i.e.
\[
\raisebox{-4.2em}{\rule{0pt}{0pt}}
\hskip -2cm
\unitlength 3mm
\linethickness{0.4pt}
\begin{picture}(20,4)(10.5,18.8)
\put(16.8,19.5){\makebox(0,0){$\delta\left(\rule{0pt}{20pt}\right.$}}
\put(22.8,19.5){\makebox(0,0){$\left.\rule{0pt}{20pt}\right)$}}
\put(20,20.4){\vector(0,1){1.7}}
\put(18,18){\vector(1,1){1.8}}
\put(22,18){\vector(-1,1){1.8}}
\put(19,18){\vector(1,2){.91}}
\put(20,20){\makebox(0,0)[cc]{\Large$\circ$}}
\put(20.5,18){\makebox(0,0)[cc]{$\ldots$}}
\put(20,16.7){\makebox(0,0)[cc]{%
   $\underbrace{\rule{12mm}{0mm}}_{\mbox{\scriptsize $k$ inputs}}$}}
\end{picture}
\hskip -2.1cm :=  \sum_{s + u = k}
\hskip -1.2cm
\unitlength3mm
\linethickness{0.4pt}
\begin{picture}(20,4)(10.5,18.8)
\put(20,20.4){\vector(0,1){1.7}}
\put(18.1,18.1){\vector(1,1){1.7}}
\put(22,18){\vector(-1,1){1.8}}
\put(19,18){\vector(1,2){.91}}
\put(20,20){\makebox(0,0)[cc]{\Large$\circ$}}
\put(20.5,18){\makebox(0,0)[cc]{$\ldots$}}
\put(20.5,16.9){\makebox(0,0)[cc]{%
   $\underbrace{\rule{9mm}{0mm}}_{\mbox{\scriptsize $s$}}$}}
\end{picture}
\put(-22.5,-2.2){
\begin{picture}(20,4)(10.5,18.8)
\put(18,18){\vector(1,1){1.8}}
\put(22,18){\vector(-1,1){1.8}}
\put(19,18){\vector(1,2){.91}}
\put(20,20){\makebox(0,0)[cc]{\Large$\circ$}}
\put(20.5,18){\makebox(0,0)[cc]{$\ldots$}}
\put(20,17){\makebox(0,0)[cc]{%
   $\underbrace{\rule{12mm}{0mm}}_{\mbox{\scriptsize $u$}}$}}
\put(21.9,19){\makebox(0,0)[cc]{$\left(\rule{0pt}{20pt} \hskip 2.2cm
    \right)_{\rm ush}$}}
\end{picture}
}\hskip -4.5em  ,\ k \geq 2,
\]
for white vertices,
\[
\raisebox{-3.2em}{\rule{0pt}{0pt}}
\unitlength 3mm
\hskip -1cm
\linethickness{0.4pt}
\delta\left(
\begin{picture}(5,3)(17.5,18.8)
\put(20,20.3){\vector(0,1){1.7}}
\put(18,18){\vector(1,1){1.8}}
\put(22,18){\vector(-1,1){1.8}}
\put(19,18){\vector(1,2){.91}}
\put(20,20){\makebox(0,0)[cc]{\Large$\bullet$}}
\put(20.5,18){\makebox(0,0)[cc]{$\ldots$}}
\put(20,16.7){\makebox(0,0)[cc]{%
   $\underbrace{\rule{12mm}{0mm}}_{\mbox{\scriptsize $k$ inputs}}$}}
\end{picture}
\right)
:= \sum_{s+u=k}
\hskip .7cm
\unitlength 3mm
\linethickness{0.4pt}
\begin{picture}(15,3)(15.5,18.8)
\put(20,20.3){\vector(0,1){1.7}}
\put(18,18){\vector(1,1){1.8}}
\put(22,18){\vector(-1,1){1.8}}
\put(19.1,18){\vector(1,2){.91}}
\put(20,20){\makebox(0,0)[cc]{\Large$\circ$}}
\put(20.5,18){\makebox(0,0)[cc]{$\ldots$}}
\put(20.5,16.8){\makebox(0,0)[cc]{%
   $\underbrace{\rule{9mm}{0mm}}_{\mbox{\scriptsize $s$}}$}}
\end{picture}
\put(-22.5,-2.2){
\begin{picture}(10,4)(10.5,18.8)
\put(18,18){\vector(1,1){1.8}}
\put(22,18){\vector(-1,1){1.8}}
\put(19.1,18){\vector(1,2){.91}}
\put(20,20){\makebox(0,0)[cc]{\Large$\bullet$}}
\put(20.5,18){\makebox(0,0)[cc]{$\ldots$}}
\put(20,16.8){\makebox(0,0)[cc]{%
   $\underbrace{\rule{11mm}{0mm}}_{\mbox{\scriptsize $u$}}$}}
\put(21.7,19){\makebox(0,0)[cc]{$\left(\rule{0pt}{17pt} \hskip 2.1cm
\right)_{\rm ush}$}}
\end{picture}
}
\hskip -2cm
-
\hskip -1cm
\unitlength 3mm
\linethickness{0.4pt}
\begin{picture}(20,3)(10.5,18.8)
\put(20,20.3){\vector(0,1){1.7}}
\put(18.1,18.1){\vector(1,1){1.7}}
\put(22,18){\vector(-1,1){1.8}}
\put(19.1,18){\vector(1,2){.91}}
\put(20,20){\makebox(0,0)[cc]{\Large$\bullet$}}
\put(20.5,18){\makebox(0,0)[cc]{$\ldots$}}
\put(20.5,16.8){\makebox(0,0)[cc]{%
   $\underbrace{\rule{9mm}{0mm}}_{\mbox{\scriptsize $s$}}$}}
\end{picture}
\put(-22.5,-2.2){
\begin{picture}(20,3)(10.5,18.8)
\put(18,18){\vector(1,1){1.8}}
\put(22,18){\vector(-1,1){1.8}}
\put(19.1,18){\vector(1,2){.91}}
\put(20,20){\makebox(0,0)[cc]{\Large$\circ$}}
\put(20.5,18){\makebox(0,0)[cc]{$\ldots$}}
\put(20,16.7){\makebox(0,0)[cc]{%
   $\underbrace{\rule{12mm}{0mm}}_{\mbox{\scriptsize $u$}}$}}
\put(21.7,19){\makebox(0,0)[cc]{$\left(\rule{0pt}{20pt} \hskip 2.1cm 
\right)_{\rm ush}$}}
\end{picture}
} \hskip -4.5em  ,\ k \geq 0,
\]
for black vertices and $\delta(\anchor) = 0$ for the anchor.  The
braces $(\hskip 1em )_{\rm ush}$ in the right hand sides
indicate the summations over all $(u,s-1)$-unshuffles.
The replacement rule for the
$\nabla$-vertices is of the form
\begin{equation}
\label{e11}
\hskip 2cm
\unitlength .5cm
\begin{picture}(.7,2)(-.5,-.5)
\put(-3.9,0){\makebox(0,0){$\delta\left(\rule{0pt}{20pt}\right.$}}
\put(1.2,0){\makebox(0,0){$\left.\rule{0pt}{20pt}\right)$}}
\put(-.5,0){\makebox(0,0)[cc]{\Large$\nabla$}}
\put(-.5,.4){\vector(0,1){1}}
\put(.1,-1.25){\vector(-1,2){.5}}
\put(.67,-1.1){\vector(-1,1){1}}
\unitlength 3mm
\put(-21,-20){
\put(19.4,18){\vector(1,4){0.475}}
\put(18.9,18){\vector(1,2){.95}}
\put(16,18){\vector(2,1){3.8}}
\put(17.75,18){\makebox(0,0)[cc]{$\ldots$}}
\put(17.7,16.6){\makebox(0,0)[cc]{%
   $\underbrace{\rule{10mm}{0mm}}_{\mbox{\scriptsize $k$ inputs}}$}}
}
\end{picture}
\hskip 1cm := \hskip .5cm
G_k \hskip .2cm - \hskip -.2cm
\unitlength 3mm
\begin{picture}(20,4)(15.5,18.8)
\put(20,20.3){\vector(0,1){1.7}}
\put(18,18){\vector(1,1){1.8}}
\put(22,18){\vector(-1,1){1.8}}
\put(19,18){\vector(1,2){.88}}
\put(20,20){\makebox(0,0)[cc]{\Large$\circ$}}
\put(20.5,18){\makebox(0,0)[cc]{$\ldots$}}
\put(20,16.6){\makebox(0,0)[cc]{%
   $\underbrace{\rule{12 mm}{0mm}}_{\mbox{\scriptsize $k+2$}}$}}
\end{picture}
\raisebox{-2.5em}{\rule{0pt}{0pt}}
\end{equation}
where $G_k$ is a linear combination of $2$-vertex trees with one
$\nabla$-vertex~(\ref{nabla}) with $u < k$, and one white
vertex~(\ref{white}) with $u < k+2$. The concrete form of $G_k$ is not
relevant for our paper, the interested reader may find some examples
in~\cite[Section~4]{janyska-markl}.  
The central statement is

\begin{theorem}[\cite{markl:na}]
\label{Th2}
Each element in $H^0(\Gr^*(d),\delta) = \Ker\left(\delta : \Gr^0(d)\to
\Gr^1(d)\right)$ represents a~natural operator $\Contor \times
T^{\otimes d} \to T$. On manifolds of dimension $\geq 2d-1$ this
correspondence is an isomorphism $H^0(\Gr^*(d),\delta) \cong
\Nat(\Contor \times T^{\otimes d},T)$.
\end{theorem}

\vskip .5em
\noindent 
{\bf Proof of Theorem~D} consists of calculation of the homology
$H^0(\Gr^*(d),\delta)$ which, of course, differs from the torsion-free
case. The first step is to observe that
$(\Gr^*(d),\delta)$ is the total complex of the following bicomplex
(see~\cite[\S XI.6]{maclane:homology} for the terminology).  For $p,q
\in {\mathbb Z}$, let
\begin{equation}
\label{bigrad}
\Gr^{p,q}(d) := \Span\left\{
\mbox {\rm graphs } \Graphname \in  \Gr^{p+q}(d);\
\mbox {\rm the number of } \nabla\mbox{\rm-vertices } = -p
\right\}.
\end{equation}
Define the horizontal differential $\deltah : \Gr^{p,q}(d) \to
\Gr^{p+1,q}(d)$ by
\begin{equation}
\label{e15}
\hskip 2.2cm
\deltah\left(\hskip 1.5cm
\unitlength .5cm
\begin{picture}(.7,2)(-.5,-.5)
\put(-.5,0){\makebox(0,0)[cc]{\Large$\nabla$}}
\put(-.5,.4){\vector(0,1){1}}
\put(.1,-1.25){\vector(-1,2){.5}}
\put(.67,-1.1){\vector(-1,1){1}}
\unitlength 3mm
\put(-21,-20){
\put(19.4,18){\vector(1,4){0.475}}
\put(18.9,18){\vector(1,2){.95}}
\put(16,18){\vector(2,1){3.8}}
\put(17.75,18){\makebox(0,0)[cc]{$\ldots$}}
\put(17.7,16.6){\makebox(0,0)[cc]{%
   $\underbrace{\rule{10mm}{0mm}}_{\mbox{\scriptsize $k$ inputs}}$}}
}
\end{picture}
\hskip .6cm
\right)
:= - \hskip -.3cm
\unitlength 3mm
\begin{picture}(20,4)(15.5,18.8)
\put(20,20.3){\vector(0,1){1.7}}
\put(18,18){\vector(1,1){1.8}}
\put(22,18){\vector(-1,1){1.8}}
\put(19,18){\vector(1,2){.88}}
\put(20,20){\makebox(0,0)[cc]{\Large$\circ$}}
\put(20.5,18){\makebox(0,0)[cc]{$\ldots$}}
\put(20,16.6){\makebox(0,0)[cc]{%
   $\underbrace{\rule{12 mm}{0mm}}_{\mbox{\scriptsize $k+2$}}$}}
\end{picture}
\raisebox{-1cm}{{\rule{0mm}{0mm}}}
\end{equation}
while $\deltah$ is trivial on remaining vertices. The vertical
differential $\deltav : \Gr^{p,q}(d) \to \Gr^{p,q+1}(d)$ is defined by
requiring that $\deltav : = \delta$ on black vertices~(\ref{black}), white
vertices~(\ref{white}) and the anchor~\anchor, while
\[
\deltav\left(\hskip 1.5cm
\unitlength .5cm
\begin{picture}(.7,2)(-.5,-.5)
\put(-.5,0){\makebox(0,0)[cc]{\Large$\nabla$}}
\put(-.5,.4){\vector(0,1){1}}
\put(.1,-1.25){\vector(-1,2){.5}}
\put(.67,-1.1){\vector(-1,1){1}}
\unitlength 3mm
\put(-21,-20){
\put(19.4,18){\vector(1,4){0.475}}
\put(18.9,18){\vector(1,2){.95}}
\put(16,18){\vector(2,1){3.8}}
\put(17.75,18){\makebox(0,0)[cc]{$\ldots$}}
\put(17.7,16.6){\makebox(0,0)[cc]{%
   $\underbrace{\rule{10mm}{0mm}}_{\mbox{\scriptsize $k$ inputs}}$}}
}
\end{picture}
\hskip .6cm
\right)
:= G_k, \hskip 3cm
\raisebox{-1.2cm}{{\rule{0mm}{0mm}}}
\]
where $G_k$ is the same as in~(\ref{e11}). We prove

\begin{lemma}
\label{l22}
The bicomplex $\Gr^{*,*}(d) = (\Gr^{*,*}(d),\deltah + \deltav)$ 
has the following properties.

\hglue .5em(i) $\Gr^{*,*}(d)$ is concentrated in the sector \ $0 \leq -p \leq q$,

\hglue .25em(ii) $\Gr^{p,*} = 0$
for $p <\!\!< 0$, and\label{ii}

(iii) the horizontal cohomology of $\Gr^{*,*}(d)$ is concentrated on the
diagonal $p+q=0$, i.e.
\[
H^p(\Gr^{*,q},\deltah) =0 \mbox{ for } p+q \not= 0 \mbox { or, equivalently, } 
H^m(\Gr^*,\deltah) = 0 \mbox { for } m \not= 0.
\]
\end{lemma}

\begin{proof} 

As in~\cite{janyska-markl}, properties (i)--(ii) follow from
simple graph combinatorics.
To verify~(iii), we follow~\cite{markl:na} and observe that
$(\Gr^*(d),\deltav)$ is a particular case of 
the following construction.
{}For each collection $(E^*,\vt) = \{(E^*(s),\vt)\}_{s
\geq 2}$ of right dg-$\Sigma_s$-modules $(E^*(s),\vt)$, one
considers the complex
$\Grd *{E^*} = (\Grd *{E^*},\vt)$ spanned by graphs with 
$d$ black vertices~(\ref{black}), one vertex \anchor\ 
and a finite number of vertices decorated by elements
of $E$. The grading of $\Grd *{E^*}$ 
is induced by the grading of $E^*$ and the differential $\vt$ replaces
$E$-decorated vertices, one at a time, by their $\vt$-images, leaving other
vertices unchanged.
Since
the assignment 
$(E^*,\vt) \mapsto (\Grd *{E^*},\vt)$ 
is an exact functor (\cite{mv}, see also~\cite[Theorem~21]{markl:ha}) , 
\[
H^*(\Grd *{E^*},\vt) \cong \Grd *{H^*(E,\vt)}.
\]

Let now $(E^*,\vt) = \{(E^*(s),\vt)\}_{s \geq 2}$ be such that
$E^0(s)$ is spanned by the symbols~(\ref{nabla}) with $u+2 = s$, $E^1(s)$
by the symbols~(\ref{white}) with $u = s$, and $E^m(s) = 0$ for $m \geq
2$. The differential $\vt$ is defined by the replacement
rule~(\ref{e15}). An equivalent description of $\vartheta : E^0(s) \to
E^1(s)$ is given on page~\pageref{a1}. It is clear that
\[
(\Gr^*(d),\deltah) \cong (\Grd *{E^*},\vt).
\]

Since $\vt : E^0(s) \to E^1(s)$ is onto, $H^*(E,\vt) =
\{H^*(E(s),\vt)\}_{s \geq 2}$ is concentrated in degree~$0$, with
$H^0(E(s),\vt)$ the kernel
\begin{equation}
\label{kernel}
\Kr(s) : = \Ker\left(\vt : E^0(s) \to E^1(s)\right).
\end{equation}
Denoting by $\Kr$ the collection $\Kr := \{\Kr(s)\}_{s \geq 2}$ we
conclude that 
\begin{equation}
\label{Jark}
H^*(\Gr^*(d),\deltah) \cong H^0(\Gr^*(d),\deltah) \cong \Gr[\Kr](d).
\end{equation}
In particular, $H^m(\Gr^*(d),\deltah) = 0$ for $m \not= 0$ which
establishes~(iii).\label{k}
\end{proof}

Properties (i)--(iii) of Lemma~\ref{l22} imply, by a standard spectral
sequence argument and the description~(\ref{Jark}) of the horizontal
cohomology, that
\[
H^0(\Gr^*(d),\delta) \cong H^0(\Gr^*(d),\deltah) \cong \Gr[\Kr](d).
\]
This, along with Theorem~\ref{Th2}, implies Theorem~D.

\vskip .5em
\noindent 
{\bf Proof of Theorem~C.}
Consider a~bicomplex $\bbB = (B^{*,*}, \delta= \deltah + \deltav)$
fulfilling (i)--(iii) of Lemma~\ref{l22} and denote by
$\Zh := \bigoplus_{r \geq 0}\Zh^r$ the (finite, by~(ii)) sum of the subspaces
\[
\Zh^r := \Ker(\deltah :B^{-r,r} \to B^{-r+1,r}). 
\]
The proposition below is a combination of Proposition~5.1 and
Corollary~5.4 of~\cite{janyska-markl}.

\begin{proposition}
\label{propo}
Let $G$ be a group and assume that the bicomplex $\bbB$ consists of
reductive $G$-modules and the differentials $\deltah$ and
$\deltav$ are $G$-equivariant. Then there exists a $G$-equivariant map
$\beta : \Zh = \bigoplus_{r \geq 0}\Zh^r \to \bigoplus_{r \geq 0}
B^{-r,r}$ such that, for each $r \geq 0$ and $z \in \Zh^r$, $\beta(z)$
is a cocycle in the total complex $\Tot(\bbB)$ of the form
$\beta(z) = z + \lot$,
with some $\lot \in \bigoplus_{p > r} B^{-p,p}$. 
\end{proposition}

By a simple spectral sequence argument, any $\beta$ as in
Proposition~\ref{propo} induces an isomorphism (denoted $\beta$ again)
\[
\beta: \Zh \stackrel{\cong}\to H^0(\Tot(\bbB)).
\]

Let $\alpha^i_{n,\sigma}$, $\sigma \in \Sigma'_n$, be
coefficients as in Theorem~A.
If we take the symbol~(\ref{nabla}), with the inputs numbered
consecutively from left to right by $\{\rada 1s\}$, as the generator of
$E^0(s)$, then $\Kr(s)$ is, as a $\Sigma_s$-module, generated by the
linear combinations
\begin{equation}
\label{xi}
\xi^i_s :=  \sum_{\sigma \in \Sigma'_s}  \alpha^i_{s,\sigma}\hskip 3.8em
\unitlength 5mm
\begin{picture}(.7,2)(-.5,-.5)
\put(-.5,0){\makebox(0,0)[cc]{\Large$\nabla$}}
\put(-.5,.4){\vector(0,1){1}}
\put(.1,-1.25){\vector(-1,2){.5}}
\put(.59,-1.22){\vector(-1,1){.97}}
\unitlength 3mm
\put(-21,-20){
\put(19.4,18){\vector(1,4){0.475}}
\put(18.9,18){\vector(1,2){.95}}
\put(16,18){\vector(2,1){3.8}}
\put(18.75,17){\makebox(0,0)[cc]{$\ldots$}}
}
\put(-6.5,-3.2){{\scriptsize $\sigma(1)$}}
\put(-.1,-3.2){{\scriptsize $\sigma(s)$}}
\end{picture}
\hskip 1.5em =   \sum_{\sigma \in \Sigma'_s}  \alpha^i_{s,\sigma} \hskip 3.2em
\begin{picture}(.7,2)(-.5,-.5)
\put(-.5,0){\makebox(0,0)[cc]{\Large$\nabla$}}
\put(-.5,.4){\vector(0,1){1}}
\put(.1,-1.25){\vector(-1,2){.5}}
\put(.59,-1.22){\vector(-1,1){.97}}
\unitlength 3mm
\put(-21,-20){
\put(19.4,18){\vector(1,4){0.475}}
\put(18.9,18){\vector(1,2){.95}}
\put(16,18){\vector(2,1){3.8}}
\put(17.75,18){\makebox(0,0)[cc]{$\ldots$}}
}
\end{picture}
\hskip .3em
\cdot \sigma\ ,\ 1 \leq i \leq k_s.
\raisebox{-2.5em}{\rule{0pt}{0pt}}
\end{equation}

For an arbitrary $k$, $0 \leq k \leq d$, denote by $\Gr^*(d)_k$ the
subspace of $\Gr^*(d)$ spanned by graphs with a~distinguished labelled
subset $\{\anteanchor1, \ldots, \anteanchor k\}$ of the set of black
vertices~(\ref{black}) with $u=0$.  There is a right $\Sigma_k$-action
on the space $\Gr^*(d)_k$ that permutes the labels of the
distinguished vertices. Let $\Gr^*_k := \bigoplus_{d \geq
k}\Gr^*(d)_k$. We wish to have, for each $(n,i) \in S$, cochains
$\varsigma^i_n \in \Gr^0(n)_n$ of the form
\begin{equation}
\label{kn}
\varsigma^i_n = \xi^i_n + \lot
\end{equation}
where $\xi^i_n$ is as in~(\ref{xi}) and l.o.t.\ a linear
combination of graphs with at least two $\nabla$-vertices. 
We also wish to have, for each $n \geq 0$, cochains
$\nu_n \in \Gr^0(n+1)_n$ of the form
\begin{equation}
\label{nn}
\nu_n = b_n + \lot
\end{equation}
where $b_n$ denotes the black vertex~(\ref{black}) with\label{gerda}
$u=n$. The abbreviation l.o.t.\ means here a linear combination of graphs
in $\Gr^0(n+1)_n$ that has at least one $\nabla$-vertex.

\begin{proposition}
\label{eqv}
There are `equivariant' cocycles 
$\{\nu_n\}_{n \geq 2}$ and $\{\varsigma^i_n\}_{(n,i) \in S}$ that enjoy the
same symmetries as the elements $\{b_n\}_{n \geq 2}$ and
$\{\xi^i_n\}_{(n,i) \in S}$.
\end{proposition}

\begin{proof}
The obvious modification of the bigrading~(\ref{bigrad}) turns
$\Gr^*(n+1)_n$ into a bicomplex satisfying conditions (i)--(iii) of
Lemma~\ref{l22}. The group $\Sigma_n$ permutes the
distinguished vertices. This action satisfies the requirements of
Proposition~\ref{propo} which therefore gives a $\Sigma_n$-equivariant
$\beta$. The element $\nu_n := \beta(b_n)$ is then the required
`equivariant' cocycle. An `equivariant' $\varsigma^i_n$ can be
constructed in the same fashion, taking $\Gr^*(n)_n$ instead of
$\Gr^*(n+1)_n$.
\end{proof}

The `ideal' tensors in Theorem~C are the natural operators
related, in the correspondence of Theorem~\ref{Th2},
to the cocycles $\{\varsigma^i_n\}_{(n,i) \in S}$ and
$\{\nu_n\}_{n \geq 2}$ constructed in Proposition~\ref{eqv}.

\vskip .5em
\noindent 
{\bf Proof of Theorem~A.}
Each iteration as in~\ref{iter}
is clearly a linear combination of terms given by contracting `free'
indexes of the local coordinate expressions of the operators
$\{D^i_n\}_{(n,i) \in S}$ and $\{V_n\}_{n \geq 2}$. Each such a contraction
is determined by a `contraction scheme,' which is a graph
with vertices of the following two types:

\begin{itemize}
\item[--] 
vertices $d^i_n$, $(n,i) \in S$, with $n$ linearly
ordered input edges and one output, and
\item[--] 
vertices $v_n$, $n \geq 0$, labeled $\rada 1d$, with $n$ linearly
ordered edges and one output.
\end{itemize}
Denote by $\Cont(d)$ the space spanned by the above contraction
schemes. One has the diagram
\begin{equation}
\label{diag}
\Gr[\Kr](d) \stackrel\pi\twoheadleftarrow \Cont(d)  \stackrel\Psi\to 
\Gr^0(d)
\end{equation}
in which the map $\pi$ replaces each vertex $d^i_n$ of a contraction
scheme $K \in \Cont(d)$ by $\xi^i_n$ defined in~(\ref{xi})
and each vertex $v_n$ by $b_n$ defined in~(\ref{black}).
The map $\Psi$ is the cocycle representing the iteration determined by
$K$.

The fact that $\pi$ is an epimorphism can be established, as
in~\cite{janyska-markl}, by constructing a right inverse
$s : \Gr[\Kr](d) \to \Cont(d)$ of $\pi$. The map $\beta = \Psi \circ
s$ has the properties as in Proposition~\ref{propo} (with trivial $G$).
It therefore induces an isomorphism $\Gr[\Kr](d) \cong
H^0(\Gr^*(d),\delta)$. In particular, the map $\Psi$ is an epimorphism
onto $\Ker(\delta : \Gr^0(d) \to \Gr^1(d)) = H^0(\Gr^*(d),\delta)$.
This, along with Theorem~\ref{Th2}, proves Theorem~A.

\vskip .5em
\noindent 
{\bf Proof of Theorem~B.}  One assigns to each graph $\Graphname \in
\Gr[\Kr](d)$ the (formal) vf-order defined by the summation
\begin{equation}
\label{eqaut}
\vforder(\Graphname) := \sum_{v \in \Vert(\Graphname)} \vforder(v),
\end{equation}
where 
\[
\vforder(v) := \cases{0}{if $v$ is $\xi^i_n$, $(n,i) \in S$, and}
                        n{if $v$ is $b_n$, $n \geq 0$.}
\]

The vf-order of a contraction scheme $G \in \Cont(d)$ can be
defined similarly, with the role of vertices $\xi_n^i$ played by
$d_n^i$, and the role of vertices $b_n$ by $v_n$.
Therefore, if a contraction scheme has vertices
$\rada{v_{p_1}}{v_{p_t}}$ for some $\Rada
p1t \geq 0$ (plus possibly some other vertices of either types), then
\begin{equation}
\label{e45}
p_1 + \cdots + p_t \leq \vforder(G).
\end{equation}

Finally, the vf-order of a graph $\Graphname$ in $\Gr^0(d)$ is
given by formula~(\ref{eqaut}) in which we define now
\[
\vforder(v) := \cases{0}{if $v$ is a $\nabla$-vertex, and}
                        n{if $v$ is $b_n$, $n \geq 0$.}
\]

The vf-order of an element of $\Gr[\Kr](d)$ (resp.~$\Cont(d)$,
resp.~$\Gr^0(d)$) is then the maximum of
vf-orders of its linear constituents. 
It is clear that the (formal) vf-order of a cocycle in $\Gr^0(d)$
equals the vf-order of the operator it represents.

As in~\cite{janyska-markl} one shows that map $\beta = \Phi \circ s :
\Gr[\Kr] \to H^0(\Gr^*(d),\delta)$ (which is an isomorphism, by the
stability assumption $\dim(M) \geq 2d-1$) constructed in the proof of
Theorem~A preserves the vf-order.  Let $\gO \in \Nat(\Contor \times
T^{\otimes d},T)$ be a differential operator represented by a cocycle
$c \in \Gr^0(d)$, $y := \beta^{-1}(c)$ and $C := s(y)$. According to
our constructions, $C \in \Cont(d)$ describes an iteration of
$\{D^i_n\}_{(n,i) \in S}$ and $\{V_n\}_{n \geq 1}$ representing $\gO$.
Since, as in~\cite{janyska-markl} both $\beta$ and $s$ preserve the
vf-order, one has $\vforder(C) = \vforder(\gO)$.  Theorem~C now
immediately follows from formula~(\ref{e45}).


\begin{thebibliography}{10}

\bibitem{Jan04}
{J. Jany\v{s}ka},
\newblock {Reduction theorems for general linear connections\/},
\newblock {Diff. Geom. Appl.} {\bf 20} (2004), 177.

\bibitem{janyska-markl}
J.~{Jany\v ska} and M.~Markl,
\newblock Combinatorial differential geometry and ideal {Bianchi}--{Ricci}
  identities.
\newblock Preprint {\tt arXiv:0809.1158}, To appear in Advances in Geometry.


\bibitem{kobayashi-nomizu}
S.~Kobayashi, K.~Nomizu,
\newblock {Foundations of Differential Geometry}, vol.~I,
\newblock Interscience Publishers, 1963.

\bibitem{kolar-michor-slovak}
I.~Kol{\'a}{\v{r}}, P.W. Michor, J.~Slov{\'a}k,
\newblock {Natural operations in differential geometry},
\newblock Springer-Verlag, Berlin, 1993.

\bibitem{KruJan90}
D.~Krupka, J.~Jany\v{s}ka,
\newblock {Lectures on differential invariants},
\newblock Folia Fac. Sci. Nat. Univ. Purkynianae Brunensis, Brno, 1990.


\bibitem{Lub72}
{G.~{\L}ubczonok},
\newblock {On reduction theorems},
\newblock Ann. Polon. Math. {\bf 26} (1972) 125--133.

\bibitem{maclane:homology}
S.~{Mac~Lane},
\newblock {Homology},
\newblock Springer-Verlag, 1963.

\bibitem{markl:ha}
M.~Markl,
\newblock Homotopy algebras are homotopy algebras,
\newblock {Forum Mathematicum} {\bf 16}(2004) 129--160.


\bibitem{markl:ig}
M.~Markl,
\newblock ${GL_n}$-invariant tensors and graphs.
\newblock {\em Archivum Math. (Brno)}, {\bf 44}( 2008) 339--353.

\bibitem{markl:na}
M.~Markl,
\newblock Natural differential operators and graph complexes,
\newblock Diff. Geom. and its Appl. {\bf 27}(2009) 257--278.


\bibitem{markl:handbook}
M.~Markl,
\newblock {Handbook of Algebra}, vol.~5, chapter Operads and {PROP}s,
  pages 87--140,
\newblock Elsevier, 2008.

\bibitem{markl-shnider-stasheff:book}
M.~Markl, S.~Shnider, J.\,D. Stasheff,
\newblock {Operads in Algebra, Topology and Physics}, vol.~96 of {
  Mathematical Surveys and Monographs},
\newblock American Mathematical Society, Providence, Rhode Island, 2002.

\bibitem{mv}
M.~Markl, A.\,A. Voronov,
\newblock {PROPped} up graph cohomology,
\newblock Preprint arXiv: {\tt math.QA/0307081}, July 2003.

\bibitem{Nij52}
A.~Nijenhuis,
\newblock Theory of the geometric object,
\newblock Thesis, University of Amsterdam, 1952.

\bibitem{Nij72}
A.~Nijenhuis,
\newblock Natural bundles and their general properties. {G}eometric objects
  revisited,
\newblock In {Differential geometry (in honor of Kentaro Yano)}, pp.
  317--334. Kinokuniya, Tokyo, 1972.

\bibitem{Sch54}
J.\,A. Schouten,
\newblock {Ricci calculus},
\newblock Berlin-G\"ottingen, 1954.

\bibitem{terng:AMJ78}
C.\,L. Terng,
\newblock Natural vector bundles and natural differential operators,
\newblock {Amer. J. Math.} {\bf 100}(1978) 775--828.

\bibitem{Veb27}
{O. Veblen},
   {Invariants of quadratic differential forms},
   Cambridge Tracts in Mathematics and Mathematical Physics No.~24,1927.       
        
\end{thebibliography}
\def\cprime{$'$}

\vfill

\end{document}